\documentclass{amsart}

\usepackage{amsfonts,amssymb}
\usepackage{mathdots,mathtools}

\newtheorem{thm}{Theorem}[section]
\newtheorem{cor}[thm]{Corollary}
\newtheorem{lem}[thm]{Lemma}
\newtheorem{prp}[thm]{Proposition}
\theoremstyle{definition}
\newtheorem{definition}[thm]{Definition}

\theoremstyle{remark}
\newtheorem{rmk}[thm]{Remark}

\allowdisplaybreaks

\newcommand{\rank}{\operatorname{rank}}

\begin{document}

\title[Canonical forms for boundary conditions]{Canonical forms for boundary conditions of self-adjoint differential operators}

\author{Yorick Hardy}
\address{School of Mathematics,
 University of the Witwatersrand, Johannesburg, Private Bag 3, Wits 2050, South Africa}
\email{yorick.hardy@wits.ac.za}
\author{Bertin Zinsou}
\address{School of Mathematics,
 University of the Witwatersrand, Johannesburg, Private Bag 3, Wits 2050, South Africa}
\email{bertin.zinsou@wits.ac.za}



\begin{abstract}
 Canonical forms of boundary conditions are important in the study
 of the eigenvalues of boundary conditions and their numerical
 computations. The known canonical forms for self-adjoint differential
 operators, with eigenvalue parameter dependent boundary conditions,
 are limited to 4-th order differential operators.  We derive
 canonical forms for self-adjoint $2n$-th order differential
 operators with eigenvalue parameter dependent boundary conditions.
 We compare the 4-th order canonical forms to the canonical forms
 derived in this article.
\end{abstract}

\maketitle

\section{Introduction}
Canonical forms of boundary conditions are important in the study of the eigenvalues of boundary conditions and their numerical computations \cite{bail}. In \cite{hao12}, Hao, Sun and Zettl investigate canonical forms of self-adjoint boundary conditions  for fourth order differential operators. They derive three mutually exclusive types of boundary conditions, which are separated, coupled and mixed boundary conditions. In \cite{bao} Bao,  Hao, Sun and Zettl provide  new canonical forms of self-adjoint boundary conditions for regular differential operators of order two and four.

In this paper, we extend the study conducted in \cite{hao12} to $2n$-th order differential operators. We start our investigation with sixth order differential operators with self-adjoint boundary conditions that we extend to $2n$-th differential operators with self-adjoint boundary conditions and we show equivalence between the separated and coupled forms presented in \cite{hao12} and those obtained during our investigation.

In Section \ref{sixth}, we introduce the self-adjoint sixth order differential operators with eigenvalue dependent boundary conditions under consideration. In Section \ref{types}, we present the types of boundary conditions for the self-adjoint sixth order differential operators. Next, using the CS-decomposition, we provide a classification of the different types of canonical forms for self-adjoint sixth order differential operators in Section \ref{cano} that  we extend in Section \ref{cano2n} to canonical forms for self-adjoint $2n$-th order differential operators. Finally, in Section \ref{revisit} we show equivalences between the separated and coupled forms provided in \cite{hao12} with those obtained in this paper.



\section{Self-adjoint sixth order boundary value problems} \label{sixth}

We consider on the interval $J=(a,b)$, $-\infty\le a<b\le \infty$, the sixth order differential equation 
with formally self-adjoint differential expression (with smooth coefficients) \cite[Remark 3.2]{MolZin15}
\begin{align}\label{sixtheq1}My=-(p_3y''')'''+(p_2y'')''+(p_1y')'+p_0y=\lambda wy,\end{align}
where $\dfrac1{p_3}$ exists on $J$, $p_j\in C^j(J)$ are sufficiently smooth real-valued functions on $J$ and
$w\in L(J,\mathbb R)$ is a real-valued Lebesgue integrable function on $J$, $w>0$ a.e. on $J$.
If the coefficients are not smooth, we introduce the quasi-derivatives
\begin{align*}&y^{[0]}=y,\ y^{[1]}=y',\ y^{[2]}=y'', \ y^{[3]}=-p_3y'''\\ &y^{[4]}=(-p_3y''')'+p_2y'', \ y^{[5]}=(-p_3y''')''+(p_2y'')'+p_1y,'\\&y^{[6]}=-(p_3y''')'''+(p_2y'')''+(p_1y')'+p_0y,\end{align*}
and \eqref{sixtheq1} is replaced by the equation $y^{[6]}=\lambda wy$
where $1/p_3, p_2, p_1, p_0,  w\in L(J,\mathbb R)$, $p_3>0$, $w>0$ a.e. on $J$
\cite[p. 3]{MolZin15}.
In either case, the boundary conditions have the same form.
Let $Y=\left(y, y', y'', y^{[3]}, y^{[4]}, y^{[5]}\right)^{\top}$. We now consider the sixth order boundary value problem defined by \eqref{sixtheq1} and the boundary conditions
\begin{align}\label{sixtheq2}AY(a)+BY(b)=0, \quad A, B\in M_6(\mathbb C).\end{align}
For the boundary conditions \eqref{sixtheq2} with the assumptions made so far, \cite[Theorem 2.4]{MolZet2} leads to  
\begin{prp}\label{sixthprop1}Let $C_6$ be the symplectic matrix of order 6 defined by \begin{align}\label{sixtheq4} C_6=\left((-1)^r\delta_{r,7-s}\right)_{r,s=1}^6,\end{align} where  $\delta$ is the Kronecker delta. Then   problems \eqref{sixtheq1}--\eqref{sixtheq2} are self-adjoint if and only if 
\begin{align}\label{sixtheq3}\rank(A:B)=6 \quad \textrm{and}\quad AC_6A^*=BC_6B^*. \end{align}\end{prp}



\section{Types of boundary conditions of sixth order differential operators}\label{types}
The following theorem gives conditions satisfied by the matrices $A$, $B$ for the problems \eqref{sixtheq1}--\eqref{sixtheq2} to be self-adjoint.
\begin{thm}\label{typesthm1}Assume that the matrices $A, B\in M_6(\mathbb C)$ satisfy \eqref{sixtheq3}. Then

{\rm (i)} $3\le \rank A\le 6$, $3\le \rank B \le 6$;

{\rm (ii)} let $0\le r \le 3$;  if $\rank A=3+r$, then $\rank B =3+r$.\end{thm}
\begin{proof} See \cite[Theorem 3]{wan1}.\end{proof}
Note that the boundary conditions \eqref{sixtheq2} are invariant under left multiplication by a non singular matrix $G\in M_6(\mathbb C)$ and if $ AC_6A^*=BC_6B^*$, then 
\begin{align*} (GA)C_6(GA)^*=(GB)C_6(GB)^*.\end{align*}
Therefore, the boundary condition form \eqref{sixtheq3} is invariant under elementary matrix row transformations of $(A:B)$.

Next, we define the different types of boundary conditions based on Theorem \ref{typesthm1}.

\begin{definition}\label{typesdef1}\rm Let the hypotheses and notation of Theorem \ref{typesthm1} hold. Then the boundary conditions \eqref{sixtheq2}, \eqref{sixtheq3} are 

\noindent (1) separated if $r=0$,\\
(2) mixed if $r=1,2$,\\
(3) coupled if $r=3$.
 \end{definition}
\begin{rmk}\label{typermk2}\rm Note that the boundary conditions \eqref{sixtheq2} are separated if each of the six boundary conditions involves only one endpoint, coupled if each of the six boundary conditions involves both endpoints, while they are mixed  if there is at least one separated and one coupled boundary conditions.  \end{rmk}

%



\section{Canonical forms for sixth order differential operators}\label{cano}

Equation \eqref{sixtheq3} can be written in the form
\begin{equation}
 \label{sixttheqdsum}
 \rank(A:B) = 6, \qquad
 (A:B)\begin{pmatrix}C_6 & 0 \\ 0 & -C_6\end{pmatrix}(A:B)^* = 0,
\end{equation}
where $\begin{pmatrix}C_6 & 0 \\ 0 & -C_6\end{pmatrix}$ is a skew-Hermitian
matrix with eigenvalues $i$ and $-i$. Thus, each column vector $\mathbf{x}_j^*$ of
$(A:B)^*$ may be written in the form
\begin{equation}
 \label{eq:imi}
 \mathbf{x}_j^* = \mathbf{x}_{j,i}^* + \mathbf{x}_{j,-i}^*
\end{equation}
where $\mathbf{x}_{j,\pm i}^*$ belongs to the eigenspace corresponding
to the eigenvalue $\pm i$. Condition \eqref{sixttheqdsum} may now be
written
\begin{equation}
 \label{eq:eqinp}
 \mathbf{x}_{j,i}\mathbf{x}_{k,i}^* = \mathbf{x}_{j,-i}\mathbf{x}_{k,-i}^*.
\end{equation}
Taking $\mathbf{x}_{j,i}$ as the rows of $X_i$ and similarly for $X_{-i}$,
\eqref{eq:imi} may be summarized as $(A:B)=X_i+X_{-i}$ and \eqref{eq:eqinp}
as 
\begin{equation}
 \label{eq:eqinpsum}
 X_iX_i^*=X_{-i}X_{-i}^*.
\end{equation}
Now decompose
\begin{equation}
 \label{eq:decompi}
 \begin{pmatrix}C_6 & 0 \\ 0 & -C_6\end{pmatrix}
 =
 V\begin{pmatrix}iI_6 & 0 \\ 0 & -iI_6\end{pmatrix}V^*
\end{equation}
where $V$ is an arbitrary unitary matrix providing the diagonalization.
From the ordering of eigenvectors (columns of $V$) in \eqref{eq:decompi} and
the solution \eqref{eq:imi} in terms  of eigenvectors, the matrix $V$
may be chosen so that $(A:B)$ has the form
\begin{equation}
 \label{eq:ABcanon}
 (A:B) = (C : D)V^*.
\end{equation}
Writing
\begin{equation*}
 V =
 \begin{pmatrix}
  \mathbf{y}_{1,i}^* &
  \cdots &
  \mathbf{y}_{6,i}^* &
  \mathbf{y}_{1,-i}^* &
  \cdots &
  \mathbf{y}_{6,-i}^*
 \end{pmatrix},\qquad
 V^* =
 \begin{pmatrix}
  \mathbf{y}_{1,i} \\
  \vdots \\
  \mathbf{y}_{6,i} \\
  \mathbf{y}_{1,-i} \\
  \vdots \\
  \mathbf{y}_{6,-i} \\
 \end{pmatrix},
\end{equation*}
where $V$ is unitary and each $\mathbf{y}_{j,\pm i}$ is an eigenvector corresponding
to the eigenvalue $\pm i$. Equation \eqref{eq:eqinpsum} yields
\begin{equation*}
 X_i
 =
 C
 \begin{pmatrix}
  \mathbf{y}_{1,i} \\
  \vdots \\
  \mathbf{y}_{6,i}
 \end{pmatrix},\qquad
 X_{-i}
 =
 D
 \begin{pmatrix}
  \mathbf{y}_{1,-i} \\
  \vdots \\
  \mathbf{y}_{6,-i}
 \end{pmatrix}
\end{equation*}
so that \eqref{eq:eqinpsum} becomes
\begin{equation*}
 CC^*=DD^*.
\end{equation*}
and, since positive definite square roots are unique,   the singular value decompositions $C=U_C\Sigma_C V_C^*$ and $D=U_D\Sigma_D V_D^*$
show that
\begin{equation*}
 \Sigma_C=(U_C^*U_D)\Sigma_D(U_C^*U_D)^*.
\end{equation*}
Hence
\begin{align*}
 (A:B)&=(U_C\Sigma_C:U_D\Sigma_D)\begin{pmatrix}V_C& 0\\ 0 & V_D\end{pmatrix}^*V^* \\
      &=U_C(\Sigma_C:U_C^*U_D\Sigma_D)\begin{pmatrix}V_C& 0\\ 0 & V_D\end{pmatrix}^*V^* \\
      &=U_C(\Sigma_C:\Sigma_C)\begin{pmatrix}V_C& 0\\ 0 & V_D(U_C^*U_D)^*\end{pmatrix}^*V^* \\
      &=U_C\Sigma_C(I_6:I_6)\begin{pmatrix}V_C& 0\\ 0 & V_D(U_C^*U_D)^*\end{pmatrix}^*V^*
\end{align*}
yields the solution \eqref{eq:imi}  and satisfies \eqref{eq:eqinp}.
Since $\rank(A:B)=6$, we have $\rank(\Sigma_C)=6$ and hence $\Sigma_C$ is invertible.
By invariance of the boundary conditions under elementary row operations, we obtain
the general form
\begin{equation}
 \label{eq:gencanon}
 (A:B)=(I_6:I_6)\begin{pmatrix}V_X& 0\\ 0 & V_Y\end{pmatrix}^*V^*
\end{equation}
where $V_X$ and $V_Y$ are arbitrary unitary matrices. Here, the first
6 columns of $V$ are eigenvectors corresponding to the eigenvalue $i$ of $C_6\oplus(-C_6)$,
and the remaining 6 columns correspond to the eigenvalue $-i$.
We write $V$ as the block matrix
\begin{equation*}
 V = \begin{pmatrix} V_{11} & V_{12} \\ V_{21} & V_{22} \end{pmatrix}
\end{equation*}
so that \eqref{eq:gencanon} becomes
\begin{equation*}
 (A:B) = (V_X^*V_{11}^*+V_Y^*V_{12}^* : V_X^*V_{21}^* + V_Y^*V_{22}^*)
\end{equation*}
where
\begin{equation*}
 \begin{pmatrix}C_6 & 0 \\ 0 & -C_6\end{pmatrix}
 \begin{pmatrix}V_{11} \\ V_{21}\end{pmatrix}
 =
 i\begin{pmatrix}V_{11} \\ V_{21}\end{pmatrix},
 \qquad
 \begin{pmatrix}C_6 & 0 \\ 0 & -C_6\end{pmatrix}
 \begin{pmatrix}V_{12} \\ V_{22}\end{pmatrix}
 =
 -i\begin{pmatrix}V_{12} \\ V_{22}\end{pmatrix}.
\end{equation*}
Again, since the boundary conditions are invariant under row operations, we will assume
\begin{equation}
 \label{eq:finalcanon}
 (A:B) = (V_{11}^*+WV_{12}^* : V_{21}^* + WV_{22}^*)
\end{equation}
where $W=V_XV_Y^*$ is unitary.
Choosing a particular $V$ provides some additional insight. For the purpose
of illustration, we also set $W=I_6$ in the following example. Let
\begin{equation*}
 V_i = \frac1{\sqrt2}
 \begin{pmatrix}
  1 & 0 & 0 \\
  0 & 1 & 0 \\
  0 & 0 & 1 \\
  0 & 0 & -i \\
  0 & i & 0 \\
  -i & 0 & 0
 \end{pmatrix}
  =
 \frac1{\sqrt2}
 \begin{pmatrix} I_3 \\ iC_3 \end{pmatrix},
\end{equation*}
\begin{equation*}
 V_{-i} = \frac1{\sqrt2}
 \begin{pmatrix}
  1 & 0 & 0 \\
  0 & 1 & 0 \\
  0 & 0 & 1 \\
  0 & 0 & i \\
  0 &-i & 0 \\
  i & 0 & 0
 \end{pmatrix}
  =
 \frac1{\sqrt2}
 \begin{pmatrix} I_3 \\ -iC_3 \end{pmatrix}
\end{equation*}
where
\begin{equation*}
 C_3 = \begin{pmatrix} 0 & 0 &-1 \\ 0 & 1 & 0 \\-1 & 0 & 0 \end{pmatrix}.
\end{equation*}
Now consider $V$ given by
\begin{equation}
 \label{eq:Vsimp}
 V =
 \begin{pmatrix}
  V_i & 0      & V_{-i} & 0 \\
  0   & V_{-i} &      0 & V_i
 \end{pmatrix}.
\end{equation}
Hence
\begin{align*}
 (A:B) &= \begin{pmatrix}
           V_i^* + V_{-i}^* & 0 \\
           0 & V_i^* + V_{-i}^*
          \end{pmatrix}
        = \begin{pmatrix}
           \sqrt2I_3 & 0 & 0 & 0\\
           0 & 0 & \sqrt2I_3 & 0
          \end{pmatrix}.
\end{align*}
Here $\text{rank}(A)=\text{rank}(B)=3$.
Choosing $V$ as above, leads to a canonical form for
separated boundary conditions in Lemma \ref{lem:6th}. \\
\indent  Let 
\begin{equation*}
 W=\begin{pmatrix} 0 & I_3 \\ I_3 & 0\end{pmatrix}.
\end{equation*}
Then
\begin{align*}
 (A:B) &= \begin{pmatrix}
            V_i^*    & V_i^* \\
            V_{-i}^* & V_{-i}^*
           \end{pmatrix}
         = \frac1{\sqrt2}
           \begin{pmatrix}
            I_3 & -iC_3 & I_3 & -iC_3 \\
            I_3 &  iC_3 & I_3 &  iC_3
           \end{pmatrix}.
\end{align*}
Here we obtain coupled boundary conditions, leading to a canonical form
in Lemma \ref{lem:6th}.

From \eqref{eq:Vsimp}, we have
\begin{equation*}
 \begin{split}
  V_{11} = \frac1{\sqrt2}\begin{pmatrix} I_3 & 0 \\  iC_3 & 0 \end{pmatrix},\qquad
  V_{12} = \frac1{\sqrt2}\begin{pmatrix} I_3 & 0 \\ -iC_3 & 0 \end{pmatrix},\\
  V_{21} = \frac1{\sqrt2}\begin{pmatrix} 0 & I_3 \\ 0 & -iC_3 \end{pmatrix},\qquad
  V_{22} = \frac1{\sqrt2}\begin{pmatrix} 0 & I_3 \\ 0 &  iC_3 \end{pmatrix}.
 \end{split}
\end{equation*}
 Choosing appropriate $W$ provides the remaining canonical forms. Thus

\begin{align*} 
A &     = \frac1{\sqrt2}\left[
       \begin{pmatrix}
        I_3 &-iC_3 \\
          0 & 0
       \end{pmatrix}
       +
       W
       \begin{pmatrix}
        I_3 &  iC_3 \\
          0 & 0
       \end{pmatrix}
      \right], \\
B &    = \frac1{\sqrt2}\left[
       \begin{pmatrix}
          0 & 0 \\
        I_3 & iC_3
       \end{pmatrix}
       +
       W
       \begin{pmatrix}
          0 & 0 \\
        I_3 & -iC_3
       \end{pmatrix}
      \right].
\end{align*}
Let
\begin{equation*}
 W = \begin{pmatrix} W_1 & W_2 \\ W_3 & W_4 \end{pmatrix}.
\end{equation*}
It follows that
\begin{align*}
 A  & = \frac1{\sqrt2}
        \begin{pmatrix} W_1 & I_3 \\ W_3 & 0 \end{pmatrix}
        \begin{pmatrix} I_3 & I_3 \\ I_3 & -I_3 \end{pmatrix}
        \begin{pmatrix} I_3 & 0 \\ 0 & iC_3 \end{pmatrix}, \\
 B  & = \frac1{\sqrt2}
        \begin{pmatrix} 0 & W_2 \\ I_3 & W_4 \end{pmatrix}
        \begin{pmatrix} I_3 & I_3 \\ I_3 & -I_3 \end{pmatrix}
        \begin{pmatrix} I_3 & 0 \\ 0 & iC_3 \end{pmatrix},
\end{align*}
and hence
\begin{align*}
 \rank(A) &= \rank(I_3)+\rank(W_3),\\
 \rank(B) &= \rank(I_3)+\rank(W_2).
\end{align*}
Necessarily $\rank(W_3)=\rank(W_2)$.
The CS-decomposition, described in detail in \cite{paige94} and \cite[Theorem 2.7,1]{horn12},
provides a useful way to speak about rank. In particular, we obtain the CS-decomposition of $W$
using \cite[Corollary 3,1]{fuhr18}
\begin{equation*}
 W = \begin{pmatrix} U_1 & 0 \\ 0 & U_2 \end{pmatrix}
     \begin{pmatrix} C & S \\ -S & C \end{pmatrix}
     \begin{pmatrix} V_1 & 0 \\ 0 & V_2 \end{pmatrix}
\end{equation*}
for some unitary matrices $U_1$, $U_2$, $V_1$ and $V_2$, and positive semi-definite
diagonal matrices $C$ and $S$ satisfying $C^2+S^2=I_3$. Hence, up to elementary row
operations,
\begin{align*}
 A  & = \frac1{\sqrt2}
        \begin{pmatrix} C & U_1^* \\ -S & 0 \end{pmatrix}
        \begin{pmatrix} V_1 & 0 \\ 0 & I_3 \end{pmatrix}
        \begin{pmatrix} I_3 & I_3 \\ I_3 & -I_3 \end{pmatrix}
        \begin{pmatrix} I_3 & 0 \\ 0 & iC_3 \end{pmatrix}, \\
 B  & = \frac1{\sqrt2}
        \begin{pmatrix} 0 & S \\ U_2^* & C \end{pmatrix}
        \begin{pmatrix} I_3 & 0 \\ 0 & V_2 \end{pmatrix}
        \begin{pmatrix} I_3 & I_3 \\ I_3 & -I_3 \end{pmatrix}
        \begin{pmatrix} I_3 & 0 \\ 0 & iC_3 \end{pmatrix}
\end{align*}
with
\begin{equation*}
 \rank(A) = \rank(B) = \rank(I_3)+\rank(S).
\end{equation*}

When $\rank{S}=0$, then
\begin{equation}
 \label{eq:sepW}
 W=\begin{pmatrix} W_1 & 0 \\ 0 & W_4 \end{pmatrix}
\end{equation}
where $W_1=U_1V_1$ and $W_2=U_2V_2$ are unitary. If $\rank{S}\neq 0$, then
$W$ does not simplify in an obvious way. Thus we have the following Lemma.

\begin{lem}
 \label{lem:6th}
 Let $A$ and $B$ be $6\times 6$ matrices satisfying 
 \begin{equation*}
  \rank(A:B)=6 \quad \textrm{and}\quad AC_6A^*=BC_6B^*.
 \end{equation*}
 Let $Z$ be the matrix
 \begin{equation*}
  Z = \frac1{\sqrt2}
      \begin{pmatrix}
       I_3 & I_3 & 0   & 0   \\
       I_3 &-I_3 & 0   & 0   \\
       0 & 0   & I_3 & I_3 \\
       0 & 0   & I_3 &-I_3 
      \end{pmatrix}
      \begin{pmatrix}
       I_3 & 0   & 0   & 0 \\
       0 &iC_3 & 0   & 0 \\
       0 & 0   & I_3 & 0 \\
       0 & 0   & 0   & iC_3
      \end{pmatrix}.
 \end{equation*}
 There exist a $6\times 6$ non singular matrix $U$,
 $3\times 3$ unitary matrices $U_1$, $U_2$, $V_1$ and $V_2$,
 and positive semi-definite diagonal matrices $C$ and $S$ with $C^2+S^2=I_3$,
 such that
 \begin{equation*}
  (A:B) = U
          \begin{pmatrix}
            C & I_3 & 0 & S \\
           -S & 0 & I_3 & C
          \end{pmatrix}
          \begin{pmatrix}
           V_1 & 0   & 0   & 0 \\
             0 & U_1^* & 0   & 0 \\
             0 & 0   & U_2^* & 0 \\
             0 & 0   & 0   & V_2 
          \end{pmatrix}
          Z
 \end{equation*}
 and, the boundary conditions are
 \begin{enumerate}
  \item separated, if and only $S=0$,
  \item mixed, if and only if $0<\rank(S)<3$.
  \item coupled, if and only if $\rank(S)=3$.
 \end{enumerate}
\end{lem}

\begin{rmk}
 It may be assumed, in this representation of $(A:B)$, that the
 diagonal entries of $C$ are non-increasing   
 and that the diagonal entries of $S$ are non-decreasing.
\end{rmk}

\section{Canonical forms for $2n$-th order differential operators}\label{cano2n}

We consider on the interval $J=(a,b)$, $-\infty\le a<b\le \infty$, the $2n$-th order differential equation 
with formally self-adjoint differential expression (with smooth coefficients) \cite[Remark 3.2]{MolZin15}
\begin{equation}\label{2ntheq}
 My=(-1)^n(p_ny^{(n)})^{(n)}+(p_{n-1}y^{(n-1)})^{(n-1)}+\cdots +(p_1y')'+p_0y=\lambda wy,
\end{equation}
where $\dfrac1{p_n}$ exists on $J$, $p_j\in C^j(J)$ and $w\in L(J,\mathbb R)$, $w>0$ a.e. on $J$.
If the coefficients are not smooth, we introduce the quasi-derivatives \cite{MolZin15}
\begin{align*}
 y^{[1]} &= y', \quad
 y^{[2]} = y'', \quad
 \ldots, \quad
 y^{[n-1]} = y^{(n-1)}, \\
 y^{[n]}   &= (-1)^n p_ny^{(n)}, \\
 y^{[n+1]} &= (-1)^n(p_ny^{(n)})'      +  p_{n-1}y^{(n-1)}, \\
 y^{[n+2]} &= (-1)^n(p_ny^{(n)})''     + (p_{n-1}y^{(n-1)})' + p_{n-2}y^{(n-2)}, \\
 &\,\,\vdots \\
 y^{[2n]}  &= (-1)^n(p_ny^{(n)})^{(n)} + (p_{n-1}y^{(n-1)})^{(n-1)} + \cdots + p_0y,
\end{align*}
and \eqref{2ntheq} is replaced by the equation $y^{[2n]}=\lambda wy$
where $1/p_n, p_{n-1}, \ldots, p_1, p_0,  w\in L(J,\mathbb R)$, $p_n>0$, $w>0$ a.e. on $J$
\cite[p. 3]{MolZin15}.
In either case, the boundary conditions have the same form.
Let $Y=\left(y^{[0]}, \ldots, y^{[2n-1]}\right)^{\top}$. We now consider the $2n$-th order boundary
value problem defined by \eqref{2ntheq} and the boundary conditions
\begin{align}\label{2nthboundary}AY(a)+BY(b)=0, \quad A, B\in M_{2n}(\mathbb C).\end{align}
For the boundary conditions \eqref{2nthboundary}, \cite[Theorem 2.4]{MolZet2} leads to  
\begin{prp}
\label{2nthprop1}Let $C_{2n}$ be the symplectic matrix of order $2n$ defined by
\begin{align}\label{2ntheq4} C_{2n}=\left((-1)^r\delta_{r,2n+1-s}\right)_{r,s=1}^{2n}.\end{align}
Then  problems \eqref{2ntheq}--\eqref{2nthboundary} are self-adjoint if and only if 
\begin{align}\label{2ntheq3}\rank(A:B)=2n \quad \textrm{and}\quad AC_{2n}A^*=BC_{2n}B^*. \end{align}
\end{prp}
The method in Section \ref{cano} generalizes in a straightforward way. Thus we obtain the following theorem.

\begin{thm}\label{thm:canon}
 Let $A$ and $B$ be $2n\times 2n$ matrices satisfying
 \begin{equation*}
  \rank(A:B)=2n \quad \textrm{and}\quad AC_{2n}A^*=BC_{2n}B^*.
 \end{equation*}
 Let $Z$ be the matrix
 \begin{equation*}
  Z = \frac1{\sqrt2}
      \begin{pmatrix}
       I_n & I_n & 0   & 0   \\
       I_n &-I_n & 0   & 0   \\
       0 & 0   & I_n & I_n \\
       0 & 0   & I_n &-I_n 
      \end{pmatrix}
      \begin{pmatrix}
       I_n & 0   & 0   & 0 \\
       0 &(-1)^{n+1}iC_n & 0   & 0 \\
       0 & 0   & I_n & 0 \\
       0 & 0   & 0   & (-1)^{n+1}iC_n
      \end{pmatrix}.
 \end{equation*}
 Then there exists a $2n\times 2n$ non singular matrix $U$ and
 $n\times n$ unitary matrices $V_1$, $U_1^*$, $U_2^*$ and $V_2$, and
 positive semi-definite diagonal matrices $C$ and $S$ with $C^2+S^2=I_n$,
 such that
 \begin{equation*}
  (A:B) = U
          \begin{pmatrix}
            C & I_n & 0 & S \\
           -S & 0 & I_n & C
          \end{pmatrix}
          \begin{pmatrix}
           V_1 & 0   & 0   & 0 \\
             0 & U_1^* & 0   & 0 \\
             0 & 0   & U_2^* & 0 \\
             0 & 0   & 0   & V_2 
          \end{pmatrix}
          Z.
 \end{equation*}
 and the boundary conditions are
 \begin{enumerate}
  \item separated, if and only if $S=0$,
  \item mixed, if and only $0<\rank(S)<n$.
  \item coupled, if and only if $\rank(S)=n$.
 \end{enumerate}
\end{thm}

\section{Revisiting canonical forms for fourth order differential operators}\label{revisit}

Hao, Sun and Zettl derived canonical forms for self-adjoint boundary conditions for 
differential equations of order four \cite{hao12}.   In this section we will show some equivalences between the canonical forms in \cite{hao12} and the forms presented in Lemma  \ref{lem:6th}. The  following canonical
forms are given \cite{hao12}. 

\begin{thm}[{\cite[Theorems 3, 4 and 5]{hao12}}]
 \label{thm:zettl}
 Let $A$ and $B$
 be $4\times 4$ matrices satisfying
 \begin{equation*}
  \rank(A:B)=4 \quad \textrm{and}\quad AC_{4}A^*=BC_{4}B^*.
 \end{equation*}
 Then, the boundary conditions are
 \begin{enumerate}
  \item separated, if there exists $4\times 4$ and $8\times 8$ non singular matrices $R$ and $R'$, respectively, such that
        \begin{equation*}
         (A:B) = R
         \begin{pmatrix}
          r_1 & \overline{a_{21}} & 0 & -1 & 0 & 0 & 0 & 0 \\
          a_{21} & r_2 & 1 & 0 & 0 & 0 & 0 & 0 \\
          0 & 0 & 0 & 0 & r_3 & \overline{b_{41}} & 0 & -1 \\
          0 & 0 & 0 & 0 & b_{41} & r_4 & 1 & 0
         \end{pmatrix}
         R'
        \end{equation*}
        for some $r_1,r_2,r_3,r_4\in\mathbb{R}$ and $a_{21},b_{41}\in\mathbb{C}$,
  \item mixed, if there exist $4\times 4$ and $8\times 8$ non singular matrix $R$ and $R'$, respectively, such that
        \begin{equation*}
         (A:B) = R
         \begin{pmatrix}
          r_1 & \overline{a_{21}} & 0 & -1 & -\overline{a_{31}} & -\overline{za_{31}} & 0 & 0 \\
          a_{21} & r_2 & 1 & 0 & -\overline{a_{32}} & -\overline{za_{32}} & 0 & 0 \\
          a_{31} & a_{32} & 0 & 0 & r_3 & \overline{b_{41}} & 0 & -1 \\
          za_{31} & za_{32} & 0 & 0 & b_{41} & r_4 & 1 & 0
         \end{pmatrix} R',
        \end{equation*}
        for some $r_1,r_2,r_3,r_4\in\mathbb{R}$ and $a_{21},a_{31},a_{32},a_{41},a_{42},b_{41}\in\mathbb{C}$,
  \item coupled, if there exist $4\times 4$ and $8\times 8$ non singular matrix $R$ and $R'$, respectively, such that
        \begin{equation*}
         (A:B) = R
         \begin{pmatrix}
          r_1 & \overline{a_{21}} & 0 & -1 & -\overline{a_{31}} & -\overline{a_{41}} & 0 & 0 \\
          a_{21} & r_2 & 1 & 0 & -\overline{a_{32}} & -\overline{a_{42}} & 0 & 0 \\
          a_{31} & a_{32} & 0 & 0 & r_3 & \overline{b_{41}} & 0 & -1 \\
          a_{41} & a_{42} & 0 & 0 & b_{41} & r_4 & 1 & 0
         \end{pmatrix}R'.
        \end{equation*}
        for some $r_1,r_2,r_3,r_4\in\mathbb{R}$ and $a_{21},a_{31},a_{32},a_{41},a_{42},b_{41}\in\mathbb{C}$,
 \end{enumerate}
\end{thm}

We will consider these forms in the context of Theorem \ref{thm:canon}.
In the case of differential equations of order four, Theorem \ref{thm:canon} becomes

\begin{cor}
 \label{cor:4th}
 Let $A$ and $B$ be $4\times 4$ matrices satisfying
 \begin{equation*}
  \rank(A:B)=4 \quad \textrm{and}\quad AC_{4}A^*=BC_{4}B^*.
 \end{equation*}
 Let $Z$ be the matrix
 \begin{equation*}
  Z = \frac1{\sqrt2}
      \begin{pmatrix}
       I_2 & I_2 & 0   & 0   \\
       I_2 &-I_2 & 0   & 0   \\
       0 & 0   & I_2 & I_2 \\
       0 & 0   & I_2 &-I_2 
      \end{pmatrix}
      \begin{pmatrix}
       I_2 & 0   & 0   & 0 \\
       0 & -iC_2 & 0   & 0 \\
       0 & 0   &  I_2  & 0 \\
       0 & 0   & 0   & -iC_2
      \end{pmatrix}.
 \end{equation*}
 Then there exists a $4\times 4$ non singular matrix $U$ and
 $2\times 2$ unitary matrices $V_1$, $U_1^*$, $U_2^*$ and $V_2$, and
 positive semi-definite diagonal matrices $C$ and $S$ with $C^2+S^2=I_2$,
 such that
 \begin{equation*}
  (A:B) = U
          \begin{pmatrix}
            C & I_2 & 0 & S \\
           -S & 0 & I_2 & C
          \end{pmatrix}
          \begin{pmatrix}
           V_1 & 0   & 0   & 0 \\
             0 & U_1^* & 0   & 0 \\
             0 & 0   & U_2^* & 0 \\
             0 & 0   & 0   & V_2 
          \end{pmatrix}
          Z.
 \end{equation*}
 and the boundary conditions are
 \begin{enumerate}
  \item separated, if and only if $S=0$,
  \item mixed, if and only if $\rank(S)=1$.
  \item coupled, if and only if $\rank(S)=2$.
 \end{enumerate}
\end{cor}

There are 36 canonical forms according to \cite[Theorem 2]{hao12}, which  yield (by elementary  operations) the forms listed in Theorem  \ref{thm:zettl}.
To show that Theorem \ref{thm:zettl} and Corollary \ref{cor:4th} are equivalent,
we need to show that $U$ (and $C$, $S$, $U_1$, $U_2$, $V_1$ and $V_2$) and $R$ and $R'$ exist for each
canonical form (i.e.\ for each
type of boundary conditions) which gives equality of the forms.
The forms given in both the theorem and the corollary ensure that $AC_4A^*=BC_4B^*$.
An exhaustive comparison of all 36 forms is too lengthy and cumbersome to  pursue here. 
We will show equivalence for two of the forms, which show clear connections
between the two representations of boundary conditions.

First, we consider separated boundary conditions, i.e. $S=0$:
\begin{equation*}
 U
 \begin{pmatrix}
   V_1 & U_1^* & 0   & 0 \\
     0 & 0   & U_2^* & V_2
 \end{pmatrix}
 Z
 =
 R
 \begin{pmatrix}
  A_{11} & C_2 &      0 &   0 \\
       0 &   0 & B_{21} & C_2
 \end{pmatrix}
 R'
\end{equation*}
where the left hand side is obtained from the separated boundary conditions form
of Corollary \ref{cor:4th} and the right hand side from the separated boundary conditions form
of Theorem \ref{thm:zettl}. Thus for example,
\begin{equation*}
 A_{11}=\begin{pmatrix}r_1&\overline{a_{21}} \\ a_{21} & r_2\end{pmatrix}
\end{equation*}
is Hermitian, and similarly for $B_{21}$. Without loss of generality we may assume $R=I_4$,
and we will also assume $R'=I_4$.
We obtain
\begin{equation*}
 U
 \begin{pmatrix}
  V_1 & U_1^* & 0   & 0 \\
    0 & 0   & U_2^* & V_2
 \end{pmatrix}
 =
 \frac1{\sqrt2}
 \begin{pmatrix}
  A_{11}+iI_2 & A_{11}-iI_2 &           0 &           0 \\
            0 &           0 & B_{21}+iI_2 & B_{21}-iI_2
 \end{pmatrix}.
\end{equation*}
From
\begin{align*}
 U\begin{pmatrix} V_1 \\ 0 \end{pmatrix} &= \frac1{\sqrt2}\begin{pmatrix} A_{11}+iI_2 \\ 0 \end{pmatrix}, &
 U\begin{pmatrix} U_1^* \\ 0 \end{pmatrix} &= \frac1{\sqrt2}\begin{pmatrix} A_{11}-iI_2 \\ 0 \end{pmatrix}, \\
 U\begin{pmatrix} 0 \\ U_2^* \end{pmatrix} &= \frac1{\sqrt2}\begin{pmatrix}  0\\  B_{21}+iI_2 \end{pmatrix}, &
 U\begin{pmatrix} 0 \\ V_2 \end{pmatrix} &= \frac1{\sqrt2}\begin{pmatrix}  0\\ B_{21}-iI_2  \end{pmatrix},
\end{align*}
we find
\begin{equation*} 
 U^{-1} = \frac{i}{\sqrt2}\begin{pmatrix} U_1^*-V_1 & 0 \\ 0 & V_2-U_2^* \end{pmatrix}.
\end{equation*}
Thus,
\begin{align*}
 \frac{i}{2}(U_1^*-V_1)(A_{11}+iI_2) &= V_1 &
 \frac{i}{2}(U_1^*-V_1)(A_{11}-iI_2) &= U_1^* \\
 \frac{i}{2}(V_2-U_2^*)(B_{21}+iI_2) &= U_2^* &
 \frac{i}{2}(V_2-U_2^*)(B_{21}-iI_2) &= V_2
\end{align*}
so that the matrices $W_1=U_1V_1$ and $W_4=U_2V_2$ in \eqref{eq:sepW} are unitary, where
\begin{align*}
 V_1^*U_1^* &= I_2-2i(A_{11}+iI_2)^{-1} = (A_{11}-iI_2)(A_{11}+iI_2)^{-1}, \\
     U_2V_2 &= I_2-2i(B_{21}+iI_2)^{-1} = (B_{21}-iI_2)(B_{21}+iI_2)^{-1},
\end{align*}
are unitary since $A_{11}$ and $B_{21}$ are Hermitian. Conversely, the matrices
\begin{align*}
 A_{11} &= -i(V_1^*U_1^*-I_2)^{-1}(V_1^*U_1^*+I_2) \\
 B_{21} &= -i(U_2V_2-I_2)^{-1}(U_2V_2+I_2) 
\end{align*}
are Hermitian whenever $V_1$, $U_1^*$, $U_2^*$ and $V_2$ are unitary.

Next, we consider coupled boundary conditions, where $\rank(S)=2$:
\begin{equation*}
 U
 \begin{pmatrix}
   CV_1 & U_1^* & 0   & SV_2 \\
  -SV_1 & 0   & U_2^* & CV_2 
 \end{pmatrix}
 Z
 =
 R
 \begin{pmatrix}
  A_{11} & C_2 &-A_{21}^* &   0 \\
  A_{21} &   0 & B_{21}   & C_2
 \end{pmatrix}
 R'
\end{equation*}
where the left hand side is obtained from the coupled boundary conditions form
of Corollary \ref{cor:4th} and the right hand side from the coupled boundary conditions form
of Theorem \ref{thm:zettl}. Hence $A_{11}$ and $B_{21}$ are Hermitian and $A_{21}$ is non singular.
Without loss of generality we may assume $R=I_4$. We will also assume $R'=I_4$. Thus
\begin{equation*}
 U
 \begin{pmatrix}
   CV_1 & U_1^* & 0   & SV_2 \\
  -SV_1 & 0   & U_2^* & CV_2 
 \end{pmatrix}
 =
 \frac1{\sqrt2}
 \begin{pmatrix}
  A_{11}+iI_2 & A_{11}-iI_2 & -A_{21}^* & -A_{21}* \\
  A_{21} &  A_{21} & B_{21}+iI_2 & B_{21}-iI_2
 \end{pmatrix}.
\end{equation*}
Since $U$ is an arbitrary invertible matrix, we will begin with the equivalent
form
\begin{multline*}
 U
 \begin{pmatrix}
   V_2^*S^{-1}CV_1 & V_2^*S^{-1}U_1^* & 0   & I_2 \\
  -I_2 & 0   & V_1^*S^{-1}U_2^* & V_1^*S^{-1}CV_2 
 \end{pmatrix}
 \\
 =
 \frac1{\sqrt2}
 \begin{pmatrix}
  A_{11}+iI_2 & A_{11}-iI_2 & -A_{21}^* & -A_{21}^* \\
  A_{21} &  A_{21} & B_{21}+iI_2 & B_{21}-iI_2
 \end{pmatrix}.
\end{multline*}
Let $U'$ be invertible such that
\begin{equation*}
 U = \begin{pmatrix}
        I_2 & A_{21}^*(B_{21}+iI_2)^{-1} \\
      -A_{21}(A_{11}-iI_2)^{-1} & I_2
     \end{pmatrix}^{-1} U'.
\end{equation*}
which is sensible since the matrix
\begin{multline*}
 \begin{pmatrix}
  I_2 & A_{21}^*(B_{21}+iI_2)^{-1} \\
  -A_{21}(A_{11}-iI_2)^{-1} & I_2
 \end{pmatrix} \\
 =
 \left(
 \begin{pmatrix}
  A_{11} & 0 \\
  0 & B_{21}
 \end{pmatrix}
 +
 \begin{pmatrix}
  -iI_2 & A_{21}^* \\
  -A_{21} & iI_2
 \end{pmatrix}
 \right)
 \begin{pmatrix}
  (A_{11}-iI_2)^{-1} & 0 \\
  0 & (B_{21}+iI_2)^{-1}
 \end{pmatrix}
\end{multline*}
is invertible since
\begin{equation*}
 \begin{pmatrix}
  -iI_2 & A_{21}^* \\
  -A_{21} & iI_2
 \end{pmatrix}
\end{equation*}
is skew-Hermitian and invertible (here we take
the block determinant \cite[Theorem 3]{silvester00} which yields a positive
definite matrix $I_2+A_{21}^*A_{21}$). Furthermore, we set
\begin{equation*}
 U' = 
    - \sqrt2 i
    \begin{pmatrix}
     (B_{21}+iI_2)(A_{21}^*)^{-1} & 0 \\
     0 & -(A_{11}-iI_2)A_{21}^{-1}
    \end{pmatrix}^{-1} U''.
\end{equation*}
It follows that,
\begin{equation*}
 U''
 \begin{pmatrix}
   V_2^*S^{-1}CV_1 & V_2^*S^{-1}U_1^* & 0   & I_2 \\
  -I_2 & 0   & V_1^*S^{-1}U_2^* & V_1^*S^{-1}CV_2 
 \end{pmatrix}
  =
 \begin{pmatrix}
  K_{11} & K_{12} & 0      & I_2 \\
  -I_2 & 0      & K_{23} & K_{24}
 \end{pmatrix}
\end{equation*}
where
\begin{align*}
 \SwapAboveDisplaySkip
 K_{11} &= \frac{i}2\left[(B_{21}+iI_2)(A_{21}^*)^{-1}(A_{11}+iI_2)+A_{21}\right],\\
 K_{12} &= \frac{i}2\left[(B_{21}+iI_2)(A_{21}^*)^{-1}(A_{11}-iI_2)+A_{21}\right],\\
 K_{23} &= -\frac{i}2\left[(A_{11}-iI_2)A_{21}^{-1}(B_{21}+iI_2)+A_{21}^*\right], \\
 K_{24} &= -\frac{i}2\left[(A_{11}-iI_2)A_{21}^{-1}(B_{21}-iI_2)+A_{21}^*\right].
\end{align*}
We must have $U''=I_4$, and since $K_{24}=K_{11}^*$ we can directly find
$V_1$, $V_2$ and $S$ and $C$ using the singular value decomposition of
$K_{11}$ and the fact that $S^2+C^2=I_2$ where $S$ and $C$ are positive semi-definite
diagonal matrices. It remains to show that unitary $U_1$ and $U_2$ exist
and satisfy the equations, i.e. if and only if
\begin{equation*}
 K_{12}K_{12}^*-K_{11}K_{24} = I_2,\qquad
 K_{23}K_{23}^*-K_{24}K_{11} = I_2.
\end{equation*}
Straight forward calculation establishes that these equalities hold. We note that $\rank(S)=2$
since $K_{12}K_{12}^*=I_2+K_{11}K_{11}^*$ is positive definite and hence $S\le I_2$.

\end{document}